\documentclass[11pt, reqno, oneside]{amsart}
\usepackage{amsmath,amsfonts,amssymb,amsthm}

\usepackage{enumerate}
\usepackage{esint}
\usepackage[pagebackref,colorlinks,citecolor=blue,linkcolor=blue]{hyperref}
\usepackage[dvipsnames]{xcolor}
\usepackage{mathtools}  

\usepackage{mathtools}
\usepackage{setspace}

\usepackage{comment}

\usepackage{geometry}
\numberwithin{equation}{section}

\theoremstyle{plain}
\newtheorem{theorem}[equation]{Theorem}

\newtheorem{lemma}[equation]{Lemma}

\theoremstyle{definition}
\newtheorem{definition}[equation]{Definition}

\newtheorem{example}[equation]{Example}

\newcommand{\Z}{{\mathbb Z}}
\newcommand{\R}{{\mathbb R}}
\newcommand{\N}{{\mathbb N}}

\newcommand{\Q}{{\mathbb Q}}

\newcommand{\Om}{\Omega}

\providecommand{\vint}[1]{\mathchoice
	{\mathop{\vrule width 5pt height 3 pt depth -2.5pt
			\kern -9pt \kern 1pt\intop}\nolimits_{\kern -5pt{#1}}}
	{\mathop{\vrule width 5pt height 3 pt depth -2.6pt
			\kern -6pt \intop}\nolimits_{\kern -3pt{#1}}}
	{\mathop{\vrule width 5pt height 3 pt depth -2.6pt
			\kern -6pt \intop}\nolimits_{\kern -3pt{#1}}}
	{\mathop{\vrule width 5pt height 3 pt depth -2.6pt
			\kern -6pt \intop}\nolimits_{\kern -3pt{#1}}}}

\newcommand{\eps}{\varepsilon}
\newcommand{\loc}{\mathrm{loc}}

\newcommand{\BV}{\mathrm{BV}}

\newcommand{\ch}{\text{\raise 1.3pt \hbox{$\chi$}\kern-0.2pt}}

\newcommand{\mres}{\mathbin{\vrule height 2ex depth 2.2pt width
		0.12ex\vrule height -0.3ex depth 2.2pt width .5ex}}

\newcommand{\ys}{y}

\DeclareMathOperator{\capa}{Cap}
\DeclareMathOperator{\rcapa}{cap}

\begin{document}
\title[Quasiconformal mappings and 
	the rank of $\tfrac {dDf}{d|Df|}$
	for $f\in \BV(\R^n; \R^n)$]{Quasiconformal mappings and \\
	the rank of $\tfrac {dDf}{d|Df|}$
for $f\in \BV(\R^n; \R^n)$}
\author{Panu Lahti}
\address{Panu Lahti,  Academy of Mathematics and Systems Science, Chinese Academy of Sciences,
	Beijing 100190, PR China, {\tt panulahti@amss.ac.cn}}


\begin{abstract}
We define a relaxed version $H_f^{\textrm{fine}}$ of the distortion number $H_f$ that is used to define quasiconformal
mappings. Then we show that for a BV function $f\in\BV(\R^n;\R^n)$,
for $|Df|$-a.e. $x\in\R^n$ it holds that $H_{f^*}^{\textrm{fine}}(x)<\infty$
if and only if $\tfrac{dDf}{d|Df|}(x)$ has full rank.
\end{abstract}

\date{\today}
\maketitle

\section{Introduction}
	
	Quasiconformal mappings are defined as homeomorphisms $f\in W_{\loc}^{1,n}(\R^n;\R^n)$, $n\ge 2$,
	for which $|\nabla f(x)|^n\le K|\det \nabla f(x)|$ for a.e. $x\in\R^n$ and a constant
	$K<\infty$. In particular, $\nabla f$ has full rank a.e.
	Quasiconformal mappings can also be defined between
	two metric spaces $(X,d_X)$ and $(Y,d_Y)$. For a mapping $f\colon X\to Y$,
	for every $x\in X$ and $r>0$ one defines
	\[
	L_f(x,r):=\sup\{d_Y(f(y),f(x))\colon d_X(y,x)\le r\}
	\]
	and
	\[
	l_f(x,r):=\inf\{d_Y(f(y),f(x))\colon d_X(y,x)\ge r\},
	\]
	and
	\[
	H_f(x,r):=\frac{L_f(x,r)}{l_f(x,r)};
	\]
	we interpret this to be $\infty$ if either the numerator is infinity or the denominator is zero.
	Then one defines
	\[
	H_f(x):=\limsup_{r\to 0} H_f(x,r).
	\]
	A homeomorphism $f\colon X\to Y$ is said to be (metric) \emph{quasiconformal} if there is a number
	$1\le H<\infty$ such that $H_f(x)\le H$ for all $x\in X$.
	In Euclidean spaces and sufficiently regular metric measure spaces,
	this is equivalent with the   ``analytic'' definition mentioned above,
	see  e.g. \cite[Theorem 9.8]{HKST}.
	
	It is natural to ask
	whether the analytic condition of $\nabla f$ having full rank
	is connected with the geometric condition of $H_f$ being finite also for more general mappings
	such as $f\in W_{\loc}^{1,1}(\R^n;\R^n)$, or even functions of bounded variation. 
	The answer is in general negative, since $L_f(x,r)$
	can easily be $\infty$ for every $x\in \R^n$ and $r>0$ already for a mapping $f\in W_{\loc}^{1,n}(\R^n;\R^n)$;
	see Example \ref{ex:W11 function}.
	With this in mind, given $U\subset \R^n$ containing $x\in \R^n$ and $r>0$, we let
	\[
	L_{f,U}(x,r):=\sup\{|f(x)-f(z)|\colon |z-x|\le r,\, z\in U\}
	\]
	and
	\[
	l_{f,U}(x,r):=\inf\{|f(x)-f(z)|\colon |z-x|\ge r,\, z\in U\},
	\]
	and then 
	\[
	H_{f,U}(x,r):=\frac{L_{f,U}(x,r)}{l_{f,U}(x,r)}\quad\textrm{and}\quad
	H_f^{\textrm{fine}}(x):=\inf_{U}\limsup_{r\to 0}H_{f,U}(x,r),
	\]
	where the infimum is taken over \emph{1-finely open sets} $U$ containing $x$.
	We give definitions in Section \ref{sec:prelis}.

Now it turns out that $H_f^{\textrm{fine}}$ and the rank of the Jacobian are connected in the following way.

\begin{theorem}\label{thm:rank}
	Let $f\in \BV(\R^n;\R^n)$, $n\ge 2$. Then for $|Df|$-a.e. $x\in\R^n$, we have
	$H_{f^*}^{\textrm{fine}}(x)<\infty$ if and only if $\tfrac{dDf}{d|Df|}(x)$ has full rank.
\end{theorem}

Here $f^*$ is the precise representative of $f$.

\section{Notation and definitions}\label{sec:prelis}

\subsection{Basic notation and definitions}

We will always work in the Euclidean space $\R^n$, $n\ge 1$.
We denote the $n$-dimensional Lebesgue measure by $\mathcal L^n$.
We denote the $s$-dimensional Hausdorff measure by $\mathcal H^{s}$, $s\ge 0$.
If a property holds outside a set of Lebesgue measure zero, we say that it holds almost everywhere,
or ``a.e.''.
With other measures, we write more explicitly e.g. ``$\mathcal H^{n-1}$-a.e.''.
The Borel $\sigma$-algebra on a set $H\subset \R^n$ is denoted by $\mathcal B(H)$.

We write $B(x,r)$ for an open ball in $\R^n$ with center $x$
and radius $r$, that is, $\{y \in \R^n \colon |y-x|<r\}$.
We denote the $\mathcal L^n$-measure of the unit ball by $\omega_n$.
We always work with the Euclidean norm $|\cdot|$ for vectors $v\in \R^n$ as well as for
matrices $A\in \R^{k\times n}$, $k\ge 1$.
For $a\in\R^k$ and $b\in\R^n$, we define the tensor product
$a\otimes b \coloneqq a b^T\in \R^{k\times n}$, where $a,b$ are considered
as column vectors.
Given $x\in \R^n$, $r>0$, and $\nu\in \R^n$ with $|\nu|=1$, we denote by
$Q_{\nu}(x,r)$ an open cube with center $x$ and side length $r$, such that $\nu$ is perpendicular to one of the faces.
Note that such a cube is not unique.
We denote by $Q(x,r)$ an open cube with sides in the directions of the coordinate axes
(such a cube is unique).

If a function $f$ is in $L^1(H;\R^k)$, for $k\in\N$ and
 for some $\mathcal L^n$-measurable set $H \subset \R^n$ of nonzero and finite Lebesgue
measure, we write
\[
f_H:=\vint{H} f(y) \,d\mathcal L^n(y) \coloneqq \frac{1}{\mathcal L^n(H)} \int_H f(y) \,d\mathcal L^n(y)
\] 
for its mean value in $H$.

We will always denote by
$\Om\subset\R^n$ an open set. Let $k\in\N$.
The Sobolev space $W^{1,1}(\Om;\R^k)$ consists of functions $f\in L^1(\Om;\R^k)$
whose first weak partial derivatives $D_lf_j$, $j=1,\ldots,k$, $l=1,\ldots,n$, belong to $L^1(\Om)$.
The Sobolev norm is
\[
\Vert f\Vert_{W^{1,1}(\Om;\R^k)}:=\Vert f\Vert_{L^1(\Om;\R^k)}+\Vert \nabla f\Vert_{L^1(\Om;\R^{k\times n})}.
\]

Let $S\subset \R^n$ be an $\mathcal H^{n-1}$-measurable set.
We say that $S$ is countably $\mathcal H^{n-1}$-rectifiable if there exist countably many Lipschitz
functions $h_j\colon \R^{n-1}\to \R^n$ such that
\[
\mathcal H^{n-1}\left(S\setminus \bigcup_{j=1}^{\infty} h_j(\R^{n-1})\right)=0.
\]
Let $H\subset \R^{n}$ be an $n-1$-dimensional hyperplane, let $\pi$ be the orthogonal projection onto $H$,
and let $\pi^{\perp}$ be the orthogonal projection onto the orthogonal complement $H^{\perp}$.
Let $h\colon H\to H^{\perp}$ be an $L$--Lipschitz function, and let
\[
S:=\{x\in\R^n\colon h(\pi(x))=\pi^{\perp}(x)\}
\]
be the graph of $h$. We call this an $L$--Lipschitz $n-1$-graph, also if $h$ is only defined on a subset of $H$.
Every countably $\mathcal H^{n-1}$-rectifiable set can be presented, modulo an $\mathcal H^{n-1}$-negligible
set, as an at most countable union of disjoint $1$--Lipschitz $n-1$-graphs, see \cite[Proposition 2.76]{AFP}.
We can also assume the $1$--Lipschitz $n-1$-graphs to be $\mathcal H^{n-1}$-measurable, since a graph defined on an 
entire  $n-1$-dimensional hyperplane is a closed set.

\subsection{Radon measures}\label{sec:Radon}
In this and the following subsections, we will follow the monograph \cite{AFP}.
Let $\Om\subset\R^n$ be an open set and $\ell \in\N$.
We  denote by $\mathcal M(\Om;\R^{\ell})$ the Banach space of vector-valued
Radon measures. For $\mu \in \mathcal M(\Om;\R^{\ell})$, the
total variation $|\mu|(\Om)$ is defined
with respect to the Euclidean norm on $\R^{\ell}$.
We further denote the set of positive measures by $\mathcal M^+(\Om)$.

For a vector-valued Radon measure $\gamma\in\mathcal M(\Om;\R^{\ell})$
and a positive Radon measure $\mu\in\mathcal M^+(\Om)$, we
can write the Lebesgue--Radon--Nikodym decomposition
\[
\gamma=\gamma^a+\gamma^s=\frac{d\gamma}{d\mu}\,d\mu+\gamma^s
\]
of $\gamma$ with respect to $\mu$,
where $\frac{d\gamma}{d\mu}\in L^1(\Om,\mu;\R^{\ell})$.

For open sets $E\subset \R^{n-m}$, $F\subset \R^m$ with $m \in \{1,\ldots,n-1\}$,
a \emph{parametrized measure} $(\nu_{\ys})_{\ys \in E}$ is a mapping from $E$ to the set
$\mathcal M(F;\R^{\ell})$ of vector-valued Radon measures on $F$. It is said to be
\emph{weakly* $\mu$-measurable}, for $\mu\in\mathcal M^+(E)$, if $\ys\mapsto \nu_{\ys}(B)$ is
$\mu$-measurable for all Borel sets $B\in\mathcal{B}(F)$ (it suffices to check this for open subsets).
Equivalently, $(\nu_{\ys})_{\ys\in E}$ is weakly* $\mu$-measurable if the function
$\ys \mapsto \int_F f(\ys,t)\,d\nu_x(t)$ is $\mu$-measurable for every bounded
$\mathcal B_\mu(E) \times \mathcal B(F)$-measurable function $f\colon E\times F\to \R$
(see \cite[Proposition 2.26]{AFP}), where $\mathcal B_\mu(E)$ denotes the $\mu$-completion of $\mathcal B(E)$.
Suppose that we additionally have
\[
\int_{E}|\nu_{\ys}|(F)\,d\mu(\ys)<\infty.
\]
In this case we denote by $\mu \otimes \nu_\ys$ the generalized product measure defined by
\begin{equation}\label{eq:product measure}
\mu\otimes\nu_\ys(A)\coloneqq \int_E \left(\int_F \ch_A(\ys,t)\,d\nu_\ys(t)\right)\,d\mu(\ys)
\end{equation}
for every $A\in \mathcal B_\mu(E) \times \mathcal B(F)$.

\subsection{Functions of bounded variation}\label{sec:BV functions}

We follow the monograph \cite[Section 3]{AFP}.
Let $k\in\N$, and as before let $\Om\subset\R^n$ be an open set. 
A function
$f\in L^1(\Omega;\R^k)$ is a function of bounded variation,
denoted $f\in \BV(\Omega;\R^k)$, if its weak derivative
is an $\R^{k\times n}$-valued Radon measure with finite total variation. This means that
there exists a (unique) Radon measure $Df$
such that for all $\varphi\in C_c^1(\Omega)$, the integration-by-parts formula
\[
\int_{\Omega}f^j\frac{\partial\varphi}{\partial y_l}\,d\mathcal L^n
=-\int_{\Omega}\varphi\,d D_l f^j,\quad j=1,\ldots,k,\ l=1,\ldots,n,
\]
holds.
If $f$ is the characteristic function of a set $E\subset\R^n$, we say that $E$ is a set of finite perimeter.

Let $f\in L^1_{\loc}(\Om)$.
The precise representative is defined by
\[
f^*(x):=\limsup_{r\to 0}\,\vint{B(x,r)}f\,d\mathcal L^n,\quad x\in \Om.
\]
This is easily seen to be a Borel function.
For $f\in L^1_{\loc}(\Om;\R^k)$, we also define $f^*:=(f_1^*,\ldots,f_k^*)$.

We say that $x\in\Om$ is a Lebesgue point of $f$ if
\[
\lim_{r\to 0}\,\vint{B(x,r)}|f(y)-\widetilde{f}(x)|\,d\mathcal L^n(y)=0
\]
for some $\widetilde{f}(x)\in\R^k$. We denote by $S_f\subset\Om$ the set where
this condition fails and call it the approximate discontinuity set.

Given $x\in\R^n$ and $r>0$, and a unit vector $\nu\in \R^n$, we define the half-balls
\begin{align*}
B_{\nu}^+(x,r)\coloneqq \{y\in B(x,r)\colon \langle y-x,\nu\rangle>0\},\\
B_{\nu}^-(x,r)\coloneqq \{y\in B(x,r)\colon \langle y-x,\nu\rangle<0\},
\end{align*}
where $\langle \cdot,\cdot\rangle$ denotes the inner product.
We say that $x\in \Om$ is an approximate jump point of $f$ if there exist a unit vector $\nu\in \R^n$
and distinct vectors $f^+(x), f^-(x)\in\R^k$ such that
\begin{equation}\label{eq:jump value 1}
\lim_{r\to 0}\,\vint{B_{\nu}^+(x,r)}|f(y)-f^+(x)|\,d\mathcal L^n(y)=0
\end{equation}
and
\begin{equation}\label{eq:jump value 2}
\lim_{r\to 0}\,\vint{B_{\nu}^-(x,r)}|f(y)-f^-(x)|\,d\mathcal L^n(y)=0.
\end{equation}
The set of all approximate jump points is denoted by $J_f$, and the functions
$f^-,f^+$ can be taken to be Borel functions.
For $f\in\BV(\Om;\R^k)$, we have that $\mathcal H^{n-1}(S_f\setminus J_f)=0$, see \cite[Theorem 3.78]{AFP}.

The lower and upper approximate limits of a function $f\in\BV_{\loc}(\Om)$
are defined respectively by
\[
f^{\wedge}(x)\coloneqq 
\sup\left\{t\in\R\colon \lim_{r\to 0}\frac{\mathcal L^n(B(x,r)\cap\{f<t\})}{\mathcal L^n(B(x,r))}=0\right\}
\]
and
\[
f^{\vee}(x)\coloneqq 
\inf\left\{t\in\R\colon \lim_{r\to 0}\frac{\mathcal L^n(B(x,r)\cap\{f>t\})}{\mathcal L^n(B(x,r))}=0\right\},
\]
for all $x\in\Om$.
We interpret the supremum and infimum of an empty set to be $-\infty$ and $\infty$, respectively.
Note that
\begin{equation}\label{eq:representatives outside jump set}
	\widetilde{f}(x)=f^{\wedge}(x)=f^{\vee}(x)
	\quad\textrm{for }x\in \Om\setminus S_f
\end{equation}
and
\begin{equation}\label{eq:representatives in jump set}
	f^{\wedge}(x)=\min\{f^{-}(x),f^+(x)\}
	\quad\textrm{and}\quad
	f^{\vee}(x)=\max\{f^{-}(x),f^+(x)\}
	\quad\textrm{for }
	x\in J_f.
\end{equation}
Given $f\in \BV(\Om;\R^k)$, for every $x\in J_f$ we moreover have
\begin{equation}\label{eq:jump point formula}
	f^*(x)=\frac{f^{-}(x)+f^{+}(x)}{2}.
\end{equation}
We write the Radon-Nikodym decomposition of the variation measure of $f$ into the absolutely continuous and singular parts with respect to $\mathcal L^n$
as
\[
Df=D^a f+D^s f=\nabla f\mathcal L^n+D^s f.
\]
Furthermore, we define the Cantor and jump parts of $Df$ by
\[
D^c f\coloneqq  D^s f\mres (\Om\setminus S_f),\qquad D^j f\coloneqq D^s f\mres J_f.
\]
Here
\[
D^s f \mres J_f(A):=D^s f (J_f\cap A),\quad \textrm{for } D^s f\textrm{-measurable } A\subset \Om.
\]
Since $\mathcal H^{n-1}(S_f\setminus J_f)=0$ and $|Df|$ vanishes on
$\mathcal H^{n-1}$-negligible sets, we get the decomposition (see \cite[Section 3.9]{AFP})
\begin{equation}\label{eq:Daf Dcg Djf}
Df=D^a f+ D^c f+ D^j f.
\end{equation}
For the jump part, we know that
\begin{equation}\label{eq:jump part representation}
	d|D^j f|=|f^{+}-f^-|\,d\mathcal H^{n-1}\mres J_f.
\end{equation}

We say that a sequence $\{f_i\}_{i=1}^{\infty}$ of functions in $\BV(\Om;\R^k)$
converges to $f\in \BV(\Om;\R^k)$ strictly in $\BV(\Om;\R^k)$ if
\[
f_i\to f\ \textrm{ in }L^1(\Om;\R^k)
\quad\textrm{and}
\quad
|Df_i|(\Om)\to |Df|(\Om)
\quad\textrm{as }i\to\infty.
\]

\subsection{One-dimensional sections of $\BV$ functions}\label{subsec:one dimensional sections}

The following notation and results on one-dimensional sections of
$\BV$ functions are given in \cite[Section 3.11]{AFP}.

Let $n=1$. Suppose $f\in\BV_{\loc}(\R)$.
We have $J_f=S_f$, $J_f$ is at most countable,
and $Df (\{x\})=0$ for every $x\in\R\setminus J_f$.
For every $x,\tilde{x}\in \R\setminus J_f$ with $x<\tilde{x}$ we have
\begin{equation}\label{eq:fundamental theorem of calculus for BV}
	f^*(\tilde{x})-f^*(x)=Df((x,\tilde{x})).
\end{equation}
For every $x,\tilde{x}\in \R$ with $x<\tilde{x}$ we have
\begin{equation}\label{eq:fundamental theorem of calculus for BV 2}
	|f^*(x)-f^*(\tilde{x})|\le |Df|([x,\tilde{x}]),
\end{equation}
and the same with $f^*$ replaced by $f^{\wedge}$ or $f^{\vee}$, or any pairing of these.

We say that $f^{\vee}(x)-f^{\wedge}(x)=|Df| (\{x\})$ is the
\emph{jump size} of $f$ at point $x$.
For a jump point $x\in J_f$ and a sequence $x_j\to x$, we have
\begin{equation}\label{eq:two sided continuity}
\lim_{j\to\infty}\min\{|f^*(x_j)-f^{\wedge}(x)|,|f^*(x_j)-f^{\vee}(x)|\}=0
\end{equation}
and the same with $f^*(x_j)$ replaced by $f^{\wedge}(x_j)$ or $f^{\vee}(x_j)$.

For a point $x=(x_1,\ldots,x_n)\in \R^n$, denote by $\pi_n$ the projection
\[
\pi_n(x):=(x_1,\ldots,x_{n-1}),
\]
and analogously we define $\pi_j$ for $j=1,\ldots,n-1$.
We also denote balls and cubes in $\R^{n-1}$ by $B_{n-1}(x,r)$ and $Q_{n-1}(x,r)$, respectively.

In $\R^n$, we denote the standard basis vectors by $e_i$, $i=1,\ldots,n$.
For a set $A \subset \R^n$ we denote the slices of $A$ by
\[
A^{n}_{z}\coloneqq \{t\in\R \colon (z,t) \in A\},\quad z\in \pi_n(A).
\]
Analogously, we define $A^{j}_{z}$ for all $j=1,\ldots,n$.
Usually we consider $A^{n}_{z}$; denote it by $A_{z}$.
For $f\in\BV(\Om;\R^k)$, we denote
\[
f_{z,n}(t):= f(z,t),\quad t\in \Om_z,\ z\in \pi_n(\Om).
\]
Again to avoid too heavy notation, write $f_{z}$ instead of $f_{z,n}$.
We know that $f_z\in \BV(\Om_z)$ 
for $\mathcal L^{n-1}$-almost every $z\in\pi_n(\Om)$  (see \cite[Theorem 3.103]{AFP}).

Analogously, we define $f_{z,l}$ for all $l=1,\ldots,n$.

Recall \eqref{eq:product measure}.
Denoting $D_l f\coloneqq \langle Df,e_l\rangle$, we further have
\[
D_l f=\mathcal L^{n-1}\otimes D f_{z,l}
\quad\textrm{and}\quad
D^j_l f=\mathcal L^{n-1}\otimes D^j f_{z,l},
\]
see \cite[Theorem 3.107 \& Theorem 3.108]{AFP}.
It follows that
\begin{equation}\label{eq:slice representation for total variation}
	|D_l f|=\mathcal L^{n-1}\otimes |D f_{z,l}| 
	\quad\textrm{and}\quad
	|D^j_l f|=\mathcal L^{n-1}\otimes |D^j f_{z,l}|,
\end{equation}
see \cite[Corollary 2.29]{AFP}. Moreover, for $\mathcal L^{n-1}$-almost every 
$z\in\pi_n(\Om)$ we have
\begin{equation}\label{eq:sections and jump sets}
	J_{f_{z}}=(J_f)_{z}\quad\textrm{and}\quad
	(f^*)_{z}(t)=(f_{z})^*(t)\ \ \textrm{for every }t\in \R\setminus J_{f_{z}},
\end{equation}
see \cite[Theorem 3.108]{AFP}.
By \cite[Theorem 3.108]{AFP} we also know that for $\mathcal L^{n-1}$-almost every 
$z\in\pi_n(\Om)$, we have
\[
\{(f_z)^{-}(t),(f_z)^{+}(t)\}=\{(f^{-})_z(t),(f^{+})_z(t)\}\quad\textrm{for every }t\in (J_f)_z.
\]
Assume $k=1$ so that $f\in \BV(\Om)$;
recalling \eqref{eq:representatives outside jump set} and \eqref{eq:representatives in jump set},
for $\mathcal L^{n-1}$-almost every $z\in\pi_n(\Om)$ we have
\begin{equation}\label{eq:upper and lower repr sections}
	(f_z)^{\wedge}(t)=(f^{\wedge})_z(t)
	\quad\textrm{and}\quad 
	(f_z)^{\vee}(t)=(f^{\vee})_z(t)
	\quad \textrm{for every }t\in \Om_z.
\end{equation}

\subsection{Capacities and fine topology}

The (Sobolev) $1$-capacity of a set $A\subset \R^n$ is defined by
\[
\capa_1(A):=\inf \Vert u\Vert_{W^{1,1}(\R^n)},
\]
where the infimum is taken over Sobolev functions $u\in W^{1,1}(\R^n)$ satisfying
$u\ge 1$ in a neighborhood of $A$.

Given sets $A\subset W\subset \R^n$, where $W$ is open, the relative $1$-capacity is defined by
\[
\rcapa_1(A,W):=\inf \int_{W}|\nabla u|\,d\mathcal L^n,
\]
where the infimum is taken over functions $u\in W_0^{1,1}(W)$ satisfying $u\ge 1$ in a neighborhood
of $A$.
The class $W_0^{1,1}(W)$ is the closure of $C^1_c(W)$ in the $W^{1,1}(\R^n)$-norm.

By \cite[Proposition 6.16]{BB}, we know that for a ball $B(x,r)$ and $A\subset B(x,r)$, we have
\begin{equation}\label{eq:capa vs rcapa}
\frac{\capa_1(A)}{1+r}
\le C\rcapa_1(A,B(x,2r)),
\end{equation}
where $C$ is a constant depending only on $n$.

For the proof of the following fact, see e.g. \cite[Lemma 2.16]{L-FC}.
\begin{lemma}\label{lem:capa in small ball}
Suppose $x\in\R^n$, $0<r<1$, and $A\subset B(x,r)$. Then we have
\[
\frac{\mathcal L^n(A)}{\mathcal L^n(B(x,r))}\le C\frac{\capa_1(A)}{r^{n-1}}
\quad\textrm{and}\quad
\rcapa_1(A,B(x,2r))\le C\capa_1(A),
\]
where $C$ is a constant depending only on $n$.
\end{lemma}

\begin{definition}\label{def:1 fine topology}
	We say that $A\subset \R^n$ is $1$-thin at the point $x\in \R^n$ if
	\[
	\lim_{r\to 0}\frac{\capa_1(A\cap B(x,r))}{r^{n-1}}=0.
	\]
	We say that a set $U\subset \R^n$ is $1$-finely open if $\R^n\setminus U$ is $1$-thin at every $x\in U$.
	Then we define the $1$-fine topology as the collection of $1$-finely open sets on $\R^n$.
\end{definition}

See also \cite[Section 4]{L-FC} for discussion on Definition \ref{def:1 fine topology}.
In fact, in \cite{L-FC} the criterion
\[
\lim_{r\to 0}\frac{\rcapa_1(A\cap B(x,r),B(x,2r))}{r^{n-1}}=0
\]
for $1$-thinness was used, in the context of more general metric measure spaces.
By \eqref{eq:capa vs rcapa}
and Lemma \ref{lem:capa in small ball},
this is equivalent with our current definition in the Euclidean setting.

\section{Preliminary results}

In this section we collect some preliminary results.

We have the following quasi-semicontinuity result from \cite[Theorem 2.5]{CDLP}.
Alternatively, see
\cite[Theorem 1.1]{LaSh} and \cite[Corollary 4.2]{L-SA}
for a proof of this result in more general metric spaces.

\begin{theorem}\label{thm:quasisemicontinuity}
	Let $\Om\subset \R^n$ be open, let $f\in\BV_{\loc}(\Om)$, and $\eps>0$.
	Then there exists an open set $G\subset\Om$
	such that $\capa_1(G)<\eps$ and $f^{\wedge}|_{\Om\setminus G}$ is finite and lower
	semicontinuous, and $f^{\vee}|_{\Om\setminus G}$ is finite and upper
	semicontinuous.
\end{theorem}

\begin{lemma}\label{lem:capacity and Hausdorff measure}
	Let $A\subset \R^n$. Then
	\[
	2\mathcal H^{n-1}(\pi_n(A))\le \capa_1(A).
	\]
\end{lemma}
\begin{proof}
	Consider $u\in W^{1,1}(\R^n)$ with $u\ge 1$ in a neighborhood of $A$.
	Then for $\mathcal L^{n-1}$-a.e. $z\in \pi_n(A)$, we have
	\[
	\int_{A_z}\left|\frac{\partial u}{\partial x_n}\right|\,ds\ge 2.
	\]
	Integrating over $\R^{n-1}$, we get
	\[
	\int_{\R^n}|\nabla u|\,dx\ge 2\mathcal H^{n-1}(\pi_n(A)).
	\]
	Thus $\Vert u\Vert_{W^{1,1}(\R^n)}\ge 2\mathcal H^{n-1}(\pi_n(A))$ and
	we get the result by taking infimum over all such $u$.
\end{proof}

\begin{lemma}\label{lem:capacity and Hausdorff measure 2}
	Let $D\subset \R^{n-1}$, let $h\colon D\to \R$ be $1$-Lipschitz, and let $S\subset\R^n$
	be the graph of $h$, that is, $S=\{x=(z,t)\in D\times\R \colon t=h(z)\}$. Then
	\[
	\mathcal H^{n-1}(S)\le (2\sqrt{2})^{n-1}\capa_1(S).
	\]
\end{lemma}
\begin{proof}
	By Lemma \ref{lem:capacity and Hausdorff measure}, we have $2\mathcal H^{n-1}(D)\le \capa_1(S)$.
	On the other hand, since $h$ is $1$-Lipschitz, it is easy to check from the definition of Hausdorff measures that
	$\mathcal H^{n-1}(S)\le (2\sqrt{2})^{n-1}\mathcal H^{n-1}(D)$, and so we get the result.
\end{proof}

\begin{lemma}\label{lem:rectifiable and cap}
	Let $S$ be a countably $\mathcal H^{n-1}$-rectifiable set with $\mathcal H^{n-1}(S)<\infty$.
	Then for $\mathcal H^{n-1}$-a.e. $x\in S$, given any $1$-finely open set $U\ni x$ we have
	\[
	\lim_{r\to 0}\frac{\mathcal H^{n-1}(B(x,r)\cap S\cap U)}{\omega_{n-1} r^{n-1}}= 1.
	\]
\end{lemma}
\begin{proof}
	We have $S=\bigcup_{j=1}^{\infty}S_j\cup N$, where each $S_j$ is a
	$\mathcal H^{n-1}$-measurable $1$-Lipschitz $n-1$-graph
	and $\mathcal H^{n-1}(N)=0$. Consider $x\in S_j$;
	we can assume that $S_j$ is the graph of $h\colon D\to \R$
	where $D\subset \R^{n-1}$ and $h$ is $1$-Lipschitz.
	Excluding $\mathcal H^{n-1}$-negligible sets, we
	can also assume that
	\[
	\lim_{r\to 0}\frac{\mathcal H^{n-1}(B(x,r)\cap S_j)}{\omega_{n-1} r^{n-1}}=1
	\quad\textrm{and}\quad
	\lim_{r\to 0}\frac{\mathcal H^{n-1}(B(x,r)\cap S\setminus S_j)}{\omega_{n-1} r^{n-1}}=0,
	\]
	see \cite[Theorem 2.56, Theorem 2.83]{AFP}.
	By Lemma \ref{lem:capacity and Hausdorff measure 2},
	\begin{align*}
		\limsup_{r\to 0}\frac{\mathcal H^{n-1}(B(x,r)\cap S_j\setminus U)}{\omega_{n-1} r^{n-1}}
		&\le (2\sqrt{n})^{n-1}\limsup_{r\to 0}\frac{\capa_1(B(x,r)\cap S_j\setminus U)}{\omega_{n-1} r^{n-1}}\\
		&=0,
	\end{align*}
	since $U$ is $1$-finely open.
	Combining these facts, we get the result.
\end{proof}

Recall Definition \ref{def:1 fine topology}.
The following theorem is proved in \cite[Theorem 3.6]{L-FC}.

\begin{theorem}\label{thm:fine diff}
	Let $f\in \BV_{\loc}(\R^n;\R^n)$. Then for a.e. $x\in \R^n$ we find a $1$-finely open set
	$U\ni x$ such that
	\[
	\lim_{U\setminus \{x\}\ni y\to x}\frac{|f^*(y)-f^*(x)-\nabla f(x)(y-x)|}{|y-x|}=0.
	\]
\end{theorem}

Note that in general we understand $\nabla f$ to be the density of the
absolutely continuous part of $Df$, but at a.e. $x\in \R^n$ we can understand
$\nabla f(x)$ to be pointwise defined according to the above theorem.

\begin{example}\label{ex:W11 function}
	Let $\{q_j\}_{j=1}^{\infty}$ be an enumeration of points in $\R^n$ with rational coordinates.
	Let $f\in W_{\loc}^{1,n}(\R^n;\R^n)$ be such that the first component function is
	\[
	f_1(x):=\sum_{j=1}^{\infty} 2^{-j}\log \max\{-\log|x-q_j|,1\},\quad x\in \R^n.
	\]
	Now clearly $L_f(x,r)=\infty$ for every $x\in \R^n$ and $r>0$.
	Thus $H_f=\infty$ everywhere, while with a suitable choice of the other component functions,
	we can have that $\nabla f$ has full rank in a set of nonzero Lebesgue measure.
	Thus these quantities are unrelated, and we are motivated to define a relaxed version of $H_f$.
\end{example}

\begin{definition}
Let $f$ be a function in $\R^n$ taking values in $\R^n$ a.e.
For an arbitrary set $U\subset \R^n$ containing $x$ and $r>0$, we let
\[
L_{f,U}(x,r):=\sup\{|f(x)-f(z)|\colon |z-x|\le r,\, z\in U\}
\]
and
\[
l_{f,U}(x,r):=\inf\{|f(x)-f(z)|\colon |z-x|\ge r,\, z\in U\},
\]
and then 
\[
H_{f,U}(x,r):=\frac{L_{f,U}(x,r)}{l_{f,U}(x,r)}\quad\textrm{and}\quad
H_f^{\textrm{fine}}(x):=\inf_{U}\limsup_{r\to 0}H_{f,U}(x,r),
\]
where the infimum is taken over $1$-finely open sets $U$ containing $x$.
We interpret $H_{f,U}(x,r)$ to be $\infty$ if the denominator is zero
or the numerator is $\infty$.
If $f(x)$ or $f(z)$ is not in $\R^n$, we interpret $|f(x)-f(z)|$ to be $\infty$.
\end{definition}

The last sentence of the definition will be needed because the precise representative $f^*$
of a function $f\in L^1_{\loc}(\R^n;\R^n)$ may take values in $\R^n$ only a.e.; we use the same interpretation
with the previously defined quantities $h_f,H_f$.

We record the following ``Borel regularity''.

\begin{lemma}\label{lem:weak Borel regularity}
	Let $\mu$ be a positive Radon measure and
	let $D\subset \R^n$. Then there exists a Borel set $D^*\supset D$ such that
	\[
	\mu(D\cap Q)=\mu(D^*\cap Q)
	\]
	for all open cubes $Q\subset \R^{n}$ (of all sizes and orientations).
\end{lemma}
\begin{proof}
	See the proof of \cite[Lemma 4.3]{L-CK}.
\end{proof}

We have the following Morse's covering theorem, see. e.g. 
\cite[Theorem 1.144]{FoLe}.

\begin{theorem}\label{thm:Morse}
	Let $A\subset \R^n$ be a bounded set, let $\mu$ be a positive Radon measure on $\R^n$,
	let $\kappa\ge 1$, and let
	\[
	\mathcal D\subset \{x+K\colon x\in A,\,K\textrm{ convex, compact}\}
	\]
	such that for all $x\in A$,
	for arbitrarily small $r>0$ there exists $x+K\in \mathcal D$ with
	$B(x,r)\subset x+K\subset B(x,\kappa r)$. Then there exists a disjoint family
	$\mathcal D'\subset \mathcal D$ with
	\[
	\mu\left(A\setminus \bigcup_{D\in \mathcal D'}D\right)=0.
	\]
\end{theorem}

\begin{lemma}\label{lem:nonmeasurable}
	Let $\mu$ be a positive Radon measure and let $D\subset \R^n$ be a set, not necessarily 
	$\mu$-measurable. Let $\alpha>0$.
	Suppose that for $\mu$-a.e. $x\in \R^n$ we find a unit vector
	$\nu\in\R^n$, a sequence $r_{j}\searrow 0$, $j\in\N$, and
	$\mu$-measurable sets $A_j\subset Q_{\nu}(x,r_j)\setminus D$ such that
	$\mu(\partial Q_{\nu}(x,r_j))=0$ and
	\[
	\frac{\mu(A_j)}{\mu(Q_{\nu}(x,r_j))}\ge \alpha.
	\]
	Then $\mu(D)=0$.
\end{lemma}

Note that the cube  $Q_{\nu}(x,r_j)$ is not unique but we can understand the notation to represent
a particular choice of such a cube for each $j$. 

\begin{proof}
	From Lemma \ref{lem:weak Borel regularity} we find 
	a Borel set $D^*\supset D$ such that
	\[
	\mu(D\cap Q)=\mu(D^*\cap Q)
	\]
	for all cubes $Q\subset \R^{n}$.
	For $\mu$-a.e. $x\in \R^n$, for all $j\in\N$ we have
	\begin{align*}
		\frac{\mu(Q_{\nu}(x,r_j)\cap D^*)}{\mu(Q_{\nu}(x,r_j))}
		&=\frac{\mu(Q_{\nu}(x,r_j)\cap D)}{\mu(Q_{\nu}(x,r_j))}\\
		&\le \frac{\mu(Q_{\nu}(x,r_j)\setminus A_j)}{\mu(Q_{\nu}(x,r_j))}\\
		&\le 1-\alpha.
	\end{align*}
	From Theorem \ref{thm:Morse} combined with the fact that $\mu(\partial Q_{\nu}(x,r_j))=0$,
	it follows that
	$\mu(D^*)=0$ and then also $\mu(D)=0$.
\end{proof}

The next two lemmas concern strict and uniform convergence of BV functions.

\begin{lemma}\label{lem:strict and uniform conv}
	Let $\Om\subset \R$ be an open interval. Let $f\in \BV(\Om)$
	with $|D^j f|(\Om)=0$ and suppose $\{f_i\}_{i=1}^{\infty}$ is a sequence
	converging to $f$  strictly in $\BV(\Om)$,
	and also $f_i^*(x)\to f^*(x)$ for $\mathcal L^1$-a.e. $x\in \Om$. Let $K\subset \Om$ be a compact interval.
	Then $f_i^*$ converges to $f^*$ uniformly in $K$.
\end{lemma}
\begin{proof}
	Fix $\eps>0$. Since $f_i^*(x)\to f^*(x)$ for a.e. $x\in \Om$
	and $|Df|$ does not charge singletons, we find a finite collection of points $x_1<\ldots<x_N$
	such that $K\subset (x_1,x_N)$, $f_i^*(x_j)\to f^*(x_j)$ for every $j=1,\ldots,N$, and
	\[
	|Df|([x_j,x_{j+1}])<\eps\quad \textrm{for all }j=1,\ldots,N-1.
	\]
	Thus for sufficiently large $i$, we have $|f_i^*(x_j)-f^*(x_j)|<\eps$ for every $j=1,\ldots,N$.
	Since $|Df|(W)\le \liminf_{i\to\infty}|Df_i|(W)$ for every open set $W\subset \Om$
	and
	\[
	|Df|(\Om)= \lim_{i\to\infty}|Df_i|(\Om), 
	\]
	we necessarily have
	\[
	|Df|(H)\ge \limsup_{i\to\infty}|Df_i|(H)
	\]
	for every compact $H\subset \Om$.
	Thus for sufficiently large $i$, we have also
	\[
	|Df_i|([x_j,x_{j+1}])<\eps\quad \textrm{for all }j=1,\ldots,N-1.
	\]
	Let $j\in \{1,\ldots,N-1\}$.
	Using \eqref{eq:fundamental theorem of calculus for BV 2},
	for all $x\in [x_j,x_{j+1}]$ we have
	\begin{align*}
	|f_i^*(x)-f^*(x)|
	&\le |f_i^*(x)-f_i^*(x_j)|+|f_i^*(x_j)-f^*(x_j)|+|f^*(x_j)-f^*(x)|\\
	&\le |Df_i|([x_j,x_{j+1}])+|f_i^*(x_j)-f^*(x_j)|+|Df|([x_j,x_{j+1}])\\
	&<\eps+\eps+\eps
	\end{align*}
for all sufficiently large $i$. In total, $|f_i^*(x)-f^*(x)|<3\eps$ for all $x\in K$, for sufficiently large $i$.
\end{proof}

\begin{lemma}\label{lem:strict and uniform conv 2}
	Let $\Om=(a,b)$ be an open interval. Let $f\in \BV(\Om)$ be increasing
	and suppose $\{f_i\}_{i=1}^{\infty}$ is a sequence
	converging to $f$ strictly in $\BV(\Om)$, and $f_i^*(x)\to f^*(x)$ for a.e. $x\in \Om$.
	Then for a.e. $t\in\Om$,
	we have
	\[
	\liminf_{i\to\infty}\inf_{s\in \Om\colon s\ge t}f^*_i(s)\ge f^*(t).
	\]
\end{lemma}
\begin{proof}
	The measures $Df_i$ can be decomposed into their positive and negative parts as
	$Df_i=(Df_i)_+-(Df_i)_-$.
	Due to the strict convergence, we have
	\[
	\lim_{i\to\infty}(Df_i)_-(\Om) = 0.
	\]
	For a.e. $t\in\Om$, we have
	$\lim_{i\to\infty}f_i^*(t)=f^*(t)$. Fix such $t$. 	
	Then by \eqref{eq:fundamental theorem of calculus for BV 2},
	\begin{align*}
		\liminf_{i\to\infty}\inf_{s\ge t}f^*_i(s)
		&\ge \liminf_{i\to\infty}f^*_i(t)
		-\limsup_{i\to\infty}(Df_i)_-([t,b))\\
		&= f^*(t).
	\end{align*}
\end{proof}

We will also need the following measurability results.

\begin{lemma}\label{lem:measurability}
	Let $f\in\BV(\R^n)$.
	Let $P\subset \R^{n-1}$ be a $\mathcal L^{n-1}$-measurable set and let
	$h_1\le h_2$ be Borel functions on $\R^{n-1}$. Then
	\[
	\bigcup_{z\in P}\{z\}\times [h_1(z),h_2(z)]
	\]
	is $\mathcal L^{n-1}\otimes |Df_z|=|D_nf|$-measurable.
\end{lemma}
\begin{proof}
There is a Borel set $P'\supset P$ such that $\mathcal L^{n-1}(P'\setminus P)=0$.
The set
\[
\bigcup_{z\in P'}\{z\}\times [h_1(z),h_2(z)]
\]
can be seen to be a Borel set e.g. as in the proof of  \cite[Lemma 2.3]{EvGa},
and the difference between the sets is contained in
\[
\bigcup_{z\in P'\setminus P}\{z\}\times (-\infty,\infty)
\]
which is of zero $\mathcal L^{n-1}\otimes |Df_z|$-measure.
\end{proof}

Recall that for a BV function $f$ on the real line,
we call $f^{\vee}(x)-f^{\wedge}(x)=|Df|(\{x\})$ the jump (size) of $f$ at point $x$.

\begin{lemma}\label{lem:biggest jump}
	Let $\Om\subset \R^n$ be open, let $k\in\N$, and let $f\in \BV(\Om;\R^k)$. 
	For every $z\in \pi_n(\Om)$ such that
	$f_z\in \BV(\Om_z;\R^k)$ (which holds for $\mathcal L^{n-1}$-a.e. $z\in  \pi_n(\Om)$),
	let $\alpha(z)$ be the size of the maximal jump.
	Then the function $\alpha$ is  $\mathcal L^{n-1}$-measurable.
\end{lemma}
\begin{proof}
	Note that $|Df_z|$ is weakly* $\mathcal L^{n-1}$-measurable;
	recall Section \ref{sec:Radon}.
	We have
	\[
	\alpha(z)= \lim_{k\to\infty}\max_{i\in\Z} |Df_z|([i2^{-k},(i+1)2^{-k}]\cap \Om_z)
	\]
	for a.e. $z\in \pi_n(\Om)$.
\end{proof}

\section{Proof of the main theorem}

In this section we prove our main theorem.

\begin{proof}[Proof of Theorem \ref{thm:rank}]
	Recall the decomposition $Df=D^a f+ D^c f+ D^j f$ from before \eqref{eq:Daf Dcg Djf}.
	We will show that the claim holds for
	$|D^a f|$-a.e.,  $|D^c f|$-a.e., and $|D^f f|$-a.e. $x\in \R^n$.\\
	
	\textbf{The absolutely continuous part.}
	We prove the claim for $|D^a f|$-a.e. $x\in\R^n$.
	Suppose $\tfrac{dDf}{d|Df|} (x)$ is of full rank.
	We need to show that $H_{f^*}^{\textrm{fine}}(x)<\infty$.
	Excluding a set of $\mathcal L^n$-measure zero, which is also a set of $|D^a f|$-measure zero,
	we have $\tfrac{dDf}{d|Df|} (x)=\frac{\nabla f(x)}{|\nabla f(x)|}$
	and the fine differentiability of Theorem \ref{thm:fine diff} holds at $x$.
	Consider the quantities
	\[
	\Vert \nabla f\Vert_{\textrm{max}}:=\max_{|v|=1}|\nabla f(x) v|
	\quad\textrm{and}\quad
	\Vert \nabla f\Vert_{\textrm{min}}:=\min_{|v|=1}|\nabla f(x) v|,
	\]
	which are both in $(0,\infty)$ since $\nabla f(x)$ has full rank.
	By Theorem \ref{thm:fine diff} we find a $1$-finely open set $V$ containing $x$ such that
	\[
	\lim_{V\setminus \{x\}\ni y\to x}\frac{|f^*(y)-f^*(x)- \nabla f(x) (y-x)|}{|y-x|}=0.
	\]
	Now letting $U:=V\cap B(x,r)$ for a sufficiently small $r>0$, we still have
	\begin{equation}\label{eq:fine diff assumption}
	\lim_{U\setminus \{x\}\ni y\to x}\frac{|f^*(y)-f^*(x)-\nabla f(x)(y-x)|}{|y-x|}=0,
	\end{equation}
	and also
	\begin{equation}\label{eq:fine diff assumption V U}
	\frac{|f^*(y)-f^*(x)-\nabla f(x)(y-x)|}{|y-x|}<\frac{\Vert \nabla f\Vert_{\textrm{min}}}{2}
	\quad\textrm{for all }y\in U\setminus \{x\}.
	\end{equation}
	Choose an arbitrary sequence $r_j\searrow 0$.
	For every $j\in\N$, also
	consider $y_j\in U\cap \overline{B}(x,r_j)$. Then we have
	\[
	\limsup_{j\to\infty}\frac{|f^*(y_j)-f^*(x)|}{r_j}
	\le \Vert \nabla f\Vert_{\textrm{max}}\quad\textrm{by }\eqref{eq:fine diff assumption}.
	\]
	Thus
	\begin{align*}
	\limsup_{j\to\infty}\frac{L_{f^*,U}(x,r_j)}{r_j}
	=\limsup_{j\to\infty}\frac{\sup\{|f^*(z)-f^*(x)|\colon |z-x|\le r_j,\, z\in U\}}{r_j}
	\le \Vert \nabla f\Vert_{\textrm{max}}.
	\end{align*}
	On the other hand, for every $j\in\N$ consider
	$\widetilde{z}_j\in U$ with $|\widetilde{z}_j-x|\ge r_j$ (if it exists,
	as it does for all sufficiently large $j$). We have
	\begin{align*}
	\liminf_{j\to\infty}\frac{|f^*(\widetilde{z}_j)-f^*(x)|}{r_j}
	&= \liminf_{j\to\infty}\frac{|f^*(\widetilde{z}_j)-f^*(x)|}{|\widetilde{z}_j-x|} \frac{|\widetilde{z}_j-x|}{r_j}
	\quad\textrm{by }\eqref{eq:fine diff assumption}\\
	&\ge \liminf_{j\to\infty}\frac{|f^*(\widetilde{z}_j)-f^*(x)|}{|\widetilde{z}_j-x|}
	\\
	&\ge \frac{\Vert \nabla f\Vert_{\textrm{min}}}{2}
	\quad\textrm{by }\eqref{eq:fine diff assumption V U}.
	\end{align*}
	Thus
	\begin{align*}
	\liminf_{j\to\infty}\frac{l_{f^*,U}(x,r_j)}{r_j}
	=\liminf_{j\to\infty}\frac{\inf\{|f^*(z)-f^*(x)|\colon |z-x|\ge r_j,\, z\in U\}}{r_j}
	\ge \frac{\Vert \nabla f\Vert_{\textrm{min}}}{2}.
	\end{align*}
	In total, we obtain
	\[
	\limsup_{j\to\infty}H_{f^*,U}(x,r_j)
	= \limsup_{j\to\infty}\frac{L_{f^*,U}(x,r_j)}{l_{f^*,U}(x,r_j)}
	\le 2\frac{\Vert \nabla f\Vert_{\textrm{max}}}{\Vert \nabla f\Vert_{\textrm{min}}}.
	\]
	We conclude
	\[
	H_{f^*}^{\textrm{fine}}(x)
	\le 2\frac{\Vert \nabla f\Vert_{\textrm{max}}}{\Vert \nabla f\Vert_{\textrm{min}}}
	<\infty.
	\]
	
	Then suppose $\tfrac{dDf}{d|Df|} (x)$ is not of full rank.
	We need to show that $H_{f^*}^{\textrm{fine}}(x)=\infty$.
	Fix an arbitrary $1$-finely open set $U$ containing $x$.
	Again excluding a set of $\mathcal L^n$-measure zero, which is also a set of $|D^a f|$-measure zero,
	we have $\tfrac{dDf}{d|Df|} (x)=\frac{\nabla f(x)}{|\nabla f(x)|}$ and
	the fine differentiability of Theorem \ref{thm:fine diff} holds at $x$.
	This means that by making the $1$-finely open set
	$U$ smaller, which only decreases $H_{f,U}(x,r)$, we have that
	\begin{equation}\label{eq:fine diff assumption 2}
	\lim_{U\setminus \{x\}\ni y\to x}\frac{|f^*(y)-f^*(x)-\nabla f(x)(y-x)|}{|y-x|}=0.
	\end{equation}
	But now the $\nabla f_j(x)$'s do not span $\R^n$; we can assume that they do not span $e_n$.
	By Lemma \ref{lem:capacity and Hausdorff measure}, we have
	\begin{align*}
	\frac{\mathcal H^{n-1}(\pi_n(\partial B(x,r)\setminus U))}{r^{n-1}}
	\le \frac{\capa_1(\partial B(x,r)\setminus U)}{r^{n-1}}
	\to 0\quad\textrm{as }r\to 0,
	\end{align*}
	since $U$ is $1$-finely open. Let $r_j\searrow 0$.
	Now we can choose points
 $z_j\in \partial B(x,r_j)\cap U$ such that the angle between $z_j-x$ and $e_n$ goes to zero. Thus
	\begin{align*}
	\limsup_{j\to\infty}\frac{l_{f^*,U}(x,r_j)}{r_j}
	&=\limsup_{j\to\infty}\frac{\inf\{|f^*(x)-f^*(z)|\colon |z-x|\ge |z_j-x|,\, z\in U\}}{r_j}\\
	&\le \limsup_{j\to\infty}\frac{|f^*(x)-f^*(z_j)|}{r_j}\\
	&= 	\limsup_{j\to\infty}\frac{| \nabla f(x) (z_j-x)|}{r_j}
	\quad\textrm{by }\eqref{eq:fine diff assumption 2}\\
	&=0.
	\end{align*}
	Similarly, we can choose points $\widetilde{z}_j\in \partial B(x,r_j)\cap U$
	such that the angle between $\widetilde{z}_j-x$ and $\nabla f_k(x)$ goes to zero,
	where we choose some nonzero $\nabla f_k(x)$. It follows that
	\begin{align*}
	\liminf_{j\to \infty}\frac{L_{f^*,U}(x,r_j)}{r_j}
	&=\liminf_{j\to \infty}\frac{\sup\{|f^*(x)-f^*(z)|\colon |z-x|\le  |\widetilde{z}_j-x|,\, z\in U\}}{r_j}\\
	&\ge \liminf_{j\to\infty}\frac{|f^*(x)-f^*(\widetilde{z}_j)|}{r_j}\\
	&=	\liminf_{j\to\infty}\frac{|\nabla f(x) (\widetilde{z}_j-x)|}{r_j}
	\quad\textrm{by }\eqref{eq:fine diff assumption 2}\\
	&\ge |\nabla f_k(x)|.
	\end{align*}
	Thus we get
	\[
	\limsup_{r\to 0}H_{f^*,U}(x,r)
	\ge \limsup_{j\to\infty}\frac{L_{f^*,U}(x,r_j)}{l_{f^*,U}(x,r_j)}
	=\infty,
	\]
	and so $H_{f^*}^{\textrm{fine}}(x)=\infty$.\\
	
	\textbf{The Cantor part.}
	By Alberti's rank one theorem, see \cite{Al},
	$\tfrac{dDf}{d|Df|} (x)$ is of rank one for $|D^c f|$-a.e. $x\in\R^n$,
	so we need to show for such $x$ that $H_{f^*}^{\textrm{fine}}(x)=\infty$.
	For such points, we have
	\[
	\frac{dDf}{d|Df|} (x)
	=\lim_{r\to 0}\frac{Df(B(x,r))}{|Df|(B(x,r))}=\eta(x)\otimes\xi(x)
	\]
	for unit vectors $\eta(x),\xi(x)\in \R^n$.
	For every such $x$ and $r>0$, we can choose a cube $Q_{\xi(x)}(x,r)$
	such that for different values of $r$ these are dilations of each other.
	Combining Morse's covering Theorem \ref{thm:Morse}
	with standard arguments, see e.g. \cite[Section 1.6]{EvGa}, we get
	\begin{equation}\label{eq:Cantor density}
		\lim_{r\to 0}\frac{|D^c f|(Q_{\xi(x)}(x,r))}{|Df|(Q_{\xi(x)}(x,r))}=1
	\end{equation}
	for $|D^c f|$-a.e. $x\in \R^n$.
	
	We define the auxiliary quantities
	\[
	H_{f^*}^{\textrm{fine},k,k'}(x):=\inf_{U}\sup_{0<r\le 1/k}H_{f^*,U}(x,r),
	\quad
	k\in\N,\ k'\in\N,
	\]
	where the infimum is taken over $1$-finely open sets $U$ containing $x$
	for which
	\[
	\frac{\capa_1(B(x,r)\setminus U)}{r^{n-1}}
	<\frac{1}{k'}
	\quad\textrm{for all }0<r\le 1/k.
	\]
	Note that for all $k'\in \N$, we have
	\[
	H_{f^*}^{\textrm{fine}}=\inf_{k\in\N}H_{f^*}^{\textrm{fine},k,k'}=\lim_{k\to\infty}H_{f^*}^{\textrm{fine},k,k'}.
	\]
	Thus it is enough to show for an arbitrary fixed $M\in\N$ that
	there exists $k'\in \N$ such that for every $k\in \N$, we have
	\begin{equation}\label{eq:reduction}
	H_{f^*}^{\textrm{fine},k,k'}(x)> M 
	\textrm{ for }|D^c f|\textrm{-a.e. }x\in \R^n.
	\end{equation}

	Fix $M\in\N$, a natural number
	\begin{equation}\label{eq:choice of k prime}
	k'\ge \frac{(10n)^{2n}}{\omega_{n-1}},
	\end{equation}
	and also $k\in\N$,
	and then
	fix $x\in \R^n$ where $\tfrac{dDf}{d|Df|} (x)$ has rank one
	and \eqref{eq:Cantor density} holds.
	We can assume that $\xi(x)=e_n$, and then we can assume that
	$Q_{\xi(x)}(x,r)$ is $Q(x,r)$,
	that is, the cube with sides in the direction of the coordinate axes.
	
For $r>0$, define the scalings
\begin{equation}\label{eq:scalings def}
f_{r}(y)=\frac{f(x+ry)-f_{Q(x,r)}}{|Df|(Q(x,r))/r^{n-1}},\quad y\in Q(0,1).
\end{equation}
Excluding another $|D f|$-negligible set,
the following blowup behavior is known, see e.g. \cite[Theorem 3.95]{AFP}.
For a suitable sequence $r_i\to 0$, we get $f_{r_i}\to w$ strictly in $\BV(Q(0,1);\R^n)$, where
\[
w(y)=\eta h(\langle y, e_n\rangle),
\]
where $h$ is an increasing, nonconstant function on $(-1/2,1/2)$.
Note that necessarily $|Dh|((-1/2,1/2))\le 1$ and $|h|\le 1$.
Now
we have that
\[
f_{r_i}\to w \quad\textrm{in }L^1(Q(0,1)) 
\quad \textrm{and}\quad
|Df_{r_i}|(Q(0,1))\to |Dw|(Q(0,1))=1,
\]
and we also have
\begin{equation}\label{eq:strict conv last component}
	|D_n f_{r_i}|(Q(0,1))\to |D_n w|(Q(0,1))=|Dw|(Q(0,1));
\end{equation}
see e.g. \cite[Proposition 3.15]{AFP}.

We can assume that $|Df|(\partial Q(x,r_i))=0$.
Passing to further subsequences if necessary,
we can also assume that $f_{r_i}\to w$ a.e. in $Q(0,1)$, and that
$(f_{r_i})_z\to w_z$ in $L^1((-1/2,1/2))$ for $\mathcal L^{n-1}$-a.e. $z\in Q_{n-1}(0,1)$.
Then by lower semicontinuity, we also have
\begin{equation}\label{eq:lower semicontinuity on slices}
	|Dw_z|((-1/2,1/2))\le \liminf_{i\to\infty}|D(f_{r_i})_z|((-1/2,1/2)).
\end{equation}
On the other hand,
\begin{equation}\label{eq:strict continuity and slices}
	\begin{split}
		&\lim_{j\to\infty}\int_{Q_{n-1}(0,1)}|D(f_{r_i})_z|((-1/2,1/2))\,d\mathcal L^{n-1}(z)\\
		&\qquad =\lim_{j\to\infty}|D_nf_{r_i}|(Q(0,1))\quad\textrm{by }\eqref{eq:slice representation for total variation}\\
		&\qquad =|D_n w|(Q(0,1))\quad\textrm{by }\eqref{eq:strict conv last component}\\
		&\qquad =\int_{Q_{n-1}(0,1)}|Dw_z|((-1/2,1/2))\,d\mathcal L^{n-1}(z)\quad\textrm{by }\eqref{eq:slice representation for total variation}.
	\end{split}
\end{equation}
Combining \eqref{eq:lower semicontinuity on slices}
and \eqref{eq:strict continuity and slices}, we necessarily have
\[
|Dw_z|((-1/2,1/2))= \lim_{i\to\infty}|D(f_{r_i})_z|((-1/2,1/2))
\]
for $\mathcal L^{n-1}$-a.e. $z\in Q_{n-1}(0,1)$, and so
\begin{equation}\label{eq:strict conv on lines}
	(f_{r_i})_z\to w_z
	\quad\textrm{strictly in }\BV((-1/2,1/2))
\end{equation}
for $\mathcal L^{n-1}$-a.e. $z\in Q_{n-1}(0,1)$.

For $U\subset \R^n$ and $y\in\R^n$, denote
\[
r_i^{-1}(U-y):=\{r_i^{-1}(y'-y)\colon y'\in U \}.
\]
Recalling that $k\in\N$ is fixed, we can assume that $r_i<1/(n k)$ for all $i\in\N$.
For all $y\in \R^n$ we can choose a $1$-finely open set
$U_y\ni y$ with
\[
\frac{\capa_1(B(y,r)\setminus U_{y})}{r^{n-1}}
<\frac{1}{k'}
\quad\textrm{for all }0<r\le 1/k
\]
and
\[
\sup_{0<r\le 1/k}H_{f^*,U_{y}}(y,r)\le H_{f^*}^{\textrm{fine},k,k'}(y)+1.
\]
In particular, for $y\in Q(0,1)$ we have
\[
\frac{\capa_1(B(x+r_iy,r)\setminus U_{x+r_iy})}{r^{n-1}}
<\frac{1}{k'}
\quad\textrm{for all }0<r\le 1/k,
\]
and
\begin{equation}\label{eq:almost optimality}
\sup_{0<r\le 1/k}H_{f^*,U_{x+r_iy}}(x+r_iy,r)\le H_{f^*}^{\textrm{fine},k,k'}(x+r_iy)+1,
\end{equation}
so that
\[
\frac{\capa_1(B(y,r)\setminus r_i^{-1}(U_{x+r_iy}-x))}{r^{n-1}}
<\frac{1}{k'}
\quad\textrm{for all }0<r\le \sqrt{n}
\]
(in fact, for all $0<r\le r_i^{-1}/k$),
so applying this with the choice $r=\sqrt{n}$,
using Lemma \ref{lem:capacity and Hausdorff measure} we get
\begin{equation}\label{eq:U comp is small initial}
	\begin{split}
		\mathcal L^{n-1}(\pi_n(Q(0,1)\setminus r_i^{-1}(U_{x+r_iy}-x)))
		&\le \frac{1}{k'}n^n.
	\end{split}
\end{equation}
Using also \eqref{eq:choice of k prime}, we get
\begin{equation}\label{eq:U comp is small}
	\begin{split}
		\mathcal L^{n-1}(\pi_n(Q(0,1)\setminus r_i^{-1}(U_{x+r_iy}-x)))
		&\le \frac{1}{4}\mathcal L^{n-1}(B_{n-1}(0,1/10)).
	\end{split}
\end{equation}

We consider two cases.\\

\textbf{Case 1: $|D^j w|(Q(0,1))=0$}.

From \eqref{eq:slice representation for total variation} we have $|D^j h|((-1/2,1/2))=0$,
and then from \eqref{eq:fundamental theorem of calculus for BV 2}
we have that $h^*$ is continuous, and then so is $w^*$.
For simplicity, let us denote by $h$, $w$ these representatives. 

From \eqref{eq:upper and lower repr sections}, we know that
$((f_{r_i})^{*})_z(t)=((f_{r_i})_z)^{*}(t)$
for $\mathcal L^{n-1}$-a.e. $z\in Q_{n-1}(0,1)$ and for every $t\in (-1/2,1/2)$.
Thus we can simply use the notation $(f_{r_i})^{*}_z(t)$, and similarly
$(f^l_{r_i})^{\wedge}_z(t)$, $(f^l_{r_i})^{\vee}_z(t)$ for the component functions
$f^l_{r_i}$, $l=1,\ldots,n$.

Since $|D^j h|((-1/2,1/2))=0$, $|Dh|$ is not concentrated on one point.
Since $h$ is nonconstant, we can assume that $|D h|((0,1/2))>0$.
Then we find $0<T_1<T_2<T_3<T_4<1/2$ such that
$|Dh|((T_1,T_2))>0$ and
$|Dh|((T_2,T_3))>0$, and $T_4-T_1<1/10$.

Recall that we denote points $y\in \R^n$ by $y=(y_1,\ldots,y_n)$.
Also recall \eqref{eq:fundamental theorem of calculus for BV}.
For some fixed $0<\eps<1/10$ and for all $y,z\in Q(0,1)$, we have
\begin{equation}\label{eq:nonconstant}
	|w(y)-w(z)|=|h(y_n)-h(z_n)|\ge \eps\quad\textrm{whenever }z_n\in (T_1,T_2),\,y_n\in (T_3,T_4).
\end{equation}
For $\mathcal L^{n-1}$-a.e. $z\in Q_{n-1}(0,1)$, by
\eqref{eq:strict conv on lines} and Lemma \ref{lem:strict and uniform conv}
we also have that 
\[
\lim_{i\to\infty}\sup_{t\in [T_1,T_4]}|(f_{r_i})^*_z(t)-w_z(t)|=0,
\]
and so for sufficiently large $i$ depending on $z$, we have
\begin{equation}\label{eq:uniform conv for T and T}
	\sup_{t\in [T_1,T_4]}|(f_{r_i})^*_z(t)-w_z(t)|
	\le \frac{\eps}{6(M+1)}
	< 
	\frac{\eps}{3}.
\end{equation}
Moreover by lower semicontinuity,
\[
\liminf_{i\to\infty}|D(f_{r_i})_z|((T_1,T_2)) \ge |Dw_z|((T_1,T_2))
=|Dh|((T_1,T_2)),
\]
so that
\begin{equation}\label{eq:strict conv slices}
	|D(f_{r_i})_z|((T_1,T_2)) \ge \frac 12 |Dh|((T_1,T_2))
\end{equation}
for sufficiently large $i$ (depending on $z$).
Denote the set of those $z\in Q_{n-1}(0,1)$ for which
\eqref{eq:uniform conv for T and T} and \eqref{eq:strict conv slices} hold by $D_i$.
This set is $\mathcal L^{n-1}$-measurable, since
$|D(f_{r_i})_z|$ is a parametrized measure, and
$(f_{r_i})^*$ and $w$ are Borel functions with
\[
\sup_{t\in [T_1,T_4]}|(f_{r_i})^*_z(t)-w_z(t)|
=\sup_{t\in [T_1,T_4]\cap \Q}|(f_{r_i})^*_z(t)-w_z(t)|
\quad\textrm{for }\mathcal L^{n-1}\textrm{-a.e. }z\in Q_{n-1}(0,1).
\]
Consider $i\in\N$ sufficiently large that
\begin{equation}\label{eq:D comp is small}
\mathcal L^{n-1}(Q_{n-1}(0,1)\setminus D_i)
\le \frac{1}{4}\mathcal L^{n-1}(B_{n-1}(0,1/10)).
\end{equation}
For such $i$, we have
\begin{equation}\label{eq:Dfrj estimate}
\begin{split}
	&|D f_{r_i}|((B_{n-1}(0,1/10)\cap D_i)\times (T_1,T_2))\\
	& \qquad \ge |D_n f_{r_i}|((B_{n-1}(0,1/10)\cap D_i)\times (T_1,T_2))\\
	&\qquad = \int_{B_{n-1}(0,1/10)\cap D_i}|D(f_{r_i})_z|((T_1,T_2))\,d\mathcal L^{n-1}(z)
	\quad\textrm{by }\eqref{eq:slice representation for total variation} \\
	&\qquad\ge \frac {\mathcal L^{n-1}(B_{n-1}(0,1/10)\cap D_i)}{2}  |Dh |((T_1,T_2))
	\quad\textrm{by }\eqref{eq:strict conv slices}\\
	&\qquad\ge  \frac {\mathcal L^{n-1}(B_{n-1}(0,1/10))}{4}  |Dh |((T_1,T_2))
	\quad\textrm{by }\eqref{eq:D comp is small}.
\end{split}
\end{equation}
Let $y\in  (B_{n-1}(0,1/10)\cap D_i) \times (T_1,T_2)$.
By \eqref{eq:D comp is small} and \eqref{eq:U comp is small}
we find a point
\[
q_y\in [(B_{n-1}(0,1/10)\cap D_i)\times (T_3,T_4)]\cap r_i^{-1}(U_{x+r_iy}-x).
\]
Since $T_4-T_1<1/10$, it follows that $|(q_y)_n-y_n|<1/10$, and
then $|q_y-y|<3/10$.
Moreover, we have
\begin{equation}\label{eq:y and q y}
\begin{split}
	|f_{r_i}^*(q_y)-f_{r_i}^*(y)| 
	&\ge \Big[-|f_{r_i}^*(q_y)-w(q_y)|
	+|w(q_y)-w(y)|
	-|w(y)-f_{r_i}^*(y)|\Big]\\
	&> -\eps/3+\eps-\eps/3\quad\textrm{by }
	\eqref{eq:nonconstant},\eqref{eq:uniform conv for T and T}\\
	&=\eps/3.
\end{split}
\end{equation}
On the other hand,
there is a point $p\in Q_{n-1}(0,1)$ such that
$B_{n-1}(p,1/10)\subset Q_{n-1}(0,1)\setminus Q_{n-1}(0,4/5)$.
Using \eqref{eq:D comp is small} and \eqref{eq:U comp is small},
we also find a point
\[
q_y'\in \big((B_{n-1}(p,1/10)\cap D_i)\times \{y_n\}\big)\cap r_i^{-1}(U_{x+r_iy}-x),
\]
and then by \eqref{eq:uniform conv for T and T},
\begin{equation}\label{eq:widehat y and y}
	\begin{split}
	|f^*_{r_i}(q_y')-f^*_{r_i}(y)|
	&\le |f^*_{r_i}(q_y')-w(q_y')|+|w(q_y')-w(y)|+|w(y)-f^*_{r_i}(y)|\\
	&\le \frac{\eps}{6(M+1)}+0+\frac{\eps}{6(M+1)}\\
	&=\frac{\eps}{3(M+1)}.
	\end{split}
\end{equation}
Note that $|q'_y-y|\ge 3/10$.
Let 
\[
A_i:=(B_{n-1}(0,1/10)\cap D_i)\times (T_1,T_2).
\]
Recall the definition of the scalings \eqref{eq:scalings def}.
For the original function $f$, for all sufficiently large $i$ we get
\begin{align*}
\frac{|Df|(x+r_i A_i)}{|Df|(Q(x,r_i))}
&= |Df_{r_i}|(A_i)\\ 
&\ge \frac {\mathcal L^{n-1}(B_{n-1}(0,1/10))}{4}  |Dh |((T_1,T_2))
\quad\textrm{by }\eqref{eq:Dfrj estimate}.
\end{align*}
Then by \eqref{eq:Cantor density}, for sufficiently large $i$ we also have
\[
\frac{|D^cf|(x+r_i A_i)}{|D^cf|(Q(x,r_i))}
\ge \frac {\mathcal L^{n-1}(B_{n-1}(0,1/10))}{5}  |Dh |((T_1,T_2)).
\]
Recall that $|q_y-y|\le 3/10$ and $|q_y'-y|\ge 3/10$.
Combining \eqref{eq:y and q y} and \eqref{eq:widehat y and y}, we get
\[
H_{f^*_{r_i},r_i^{-1}(U_{x+r_iy}-x)}(y,3/10)
\ge \frac{|f^*_{r_i}(q_y)-f^*_{r_i}(y)|}{|f^*_{r_i}(q_y')-f^*_{r_i}(y)|}
> \frac{\eps/3}{\eps/(3(M+1))}=M+1.
\]
Now
\[
H_{f^*,U_{x+r_i y}}(x+r_{i}y,3r_{i}/10)
=H_{f^*_{r_{i}},r_{i}^{-1}(U_{x+r_iy}-x)}(y,3/10)
>M+1.
\]
Then by \eqref{eq:almost optimality} we have
\begin{align*}
H_{f^*}^{\textrm{fine},k,k'}(x+r_iy)+1
&\ge \sup_{0<r\le 1/k}H_{f^*,U_{x+r_iy}}(x+r_iy,r)\\
&\ge H_{f^*,U_{x+r_i y}}(x+r_{i}y,3r_{i}/10)\\
&>M+1.
\end{align*}
Apply Lemma \ref{lem:nonmeasurable} with the choices $\mu=|D^c f|$
and $D$ is the set where
$H_f^{\textrm{fine},k,k'}\le M$.
By the lemma, this set has $|D^c f|$-measure zero.
Thus $H_f^{\textrm{fine},k,k'}(x)>M$ for $|D^c f|$-a.e. $x\in \R^n$,
which is what we needed to show (recall \eqref{eq:reduction}).\\

\textbf{Case 2: $|D^j w|(Q(0,1))>0$}.

Using \eqref{eq:slice representation for total variation}, we find that also
$|D^jh|((-1/2,1/2))>0$.
Consider $s\in (-1,1)\cap J_{h}$ such that $h$ has its biggest jump at $s$. In particular,
$h^{\vee}(s)-h^{\wedge}(s)=|Dh|(\{s\})>0$.

We can assume that $s\ge 0$.
From \eqref{eq:Cantor density}, we have
\begin{equation}\label{eq:jumps go to zero}
	\lim_{i\to\infty}|D^j f_{r_i}|(Q(0,1))=0.
\end{equation}
Recall the representation $w(y)=\eta h(y_n)$.
Pick $l\in \{1,\ldots,n\}$ with the largest value of $|\eta_l|>0$; we can assume that $l=1$ and that in fact
$\eta_1>0$.
Recall that $M\in\N$ is a (large) fixed number.
Choose $0<\delta<1/100$ sufficiently small that
\begin{equation}\label{eq:choice of delta}
\delta<\eta_1 \frac{h^{\vee}(s)-h^{\wedge}(s)}{4(M+1)}.
\end{equation}
Note that $|D^c f|$-almost all of the set of
points of the type considered in the current Case 2 can be represented as a countable
union of points where the jump size $h^{\vee}(s)-h^{\wedge}(s)$ is at least, say, $1/l'$ with $l'\in\N$.
Thus we can assume that $h^{\vee}(s)-h^{\wedge}(s)$ is at least a strictly positive number, and so
$\delta$ in fact only depends on $M$.
Then choose $0<\eps<1/10$ sufficiently small that $s+\eps<1$ and
\begin{equation}\label{eq:choice of eps}
	40^n \omega_{n-1}\omega_n^{-1}\delta^{-n}\eps^{n-1}
	\le \frac{1}{4}\mathcal L^{n-1}(B_{n-1}(0,1/10)).
\end{equation}
Apart from \eqref{eq:choice of k prime}, assume
\begin{equation}\label{eq:choice of k prime two}
k'\ge \frac{2^{4n^2}\eps^{-n}}{\omega_{n-1}},
\end{equation}
so that $k'$ only depends on $M$ (apart from the dimension $n$).
Consider the set $D_i$ of those $z \in Q_{n-1}(0,1)$ for which
\begin{equation}\label{eq:only small jumps}
	(f_{r_i})_z\in\BV((-1/2,1/2))\textrm{ has jumps at most size }
	\ \delta\min\{1, \eta_1 (h^{\vee}(s)-h^{\wedge}(s))\}.
\end{equation}
By Lemma \ref{lem:biggest jump}, $D_i$ is $\mathcal L^{n-1}$-measurable.
Due to \eqref{eq:jumps go to zero} and \eqref{eq:slice representation for total variation}, we have
\begin{equation}\label{eq:size of Dj sets}
	\lim_{i\to\infty}\mathcal L^{n-1}(D_i)=\mathcal L^{n-1}(Q_{n-1}(0,1)).
\end{equation}
By Theorem \ref{thm:quasisemicontinuity} and Lemma \ref{lem:capacity and Hausdorff measure},
we can also assume that $(f^1_{r_i})^{\wedge}$ and $(f^1_{r_i})^{\vee}$ restricted to $D_i\times (-1/2,1/2)$ 
are lower and upper semicontinuous, respectively.
Moreover, by \eqref{eq:strict conv on lines}
we can assume  that $(f_{r_i})_z\to w_z$ strictly in $\BV((-1/2,1/2))$ for all $z\in D_i$, 
and $(f_{r_i})^*_z(t_0)\to w^*_z(t_0)$ for some $t_0\in [s-\eps,s+\eps]$.
Thus by \eqref{eq:fundamental theorem of calculus for BV 2}, we get
\begin{align*}
	\limsup_{i\to\infty}\sup_{t\in [s-\eps,s+\eps]}|(f_{r_i})^{*}_z(t)-w^*_z(t_0)|
	&\le \limsup_{i\to\infty}\sup_{t,t'\in [s-\eps,s+\eps]}|(f_{r_i})^{*}_z(t)-(f_{r_i})^{*}_z(t')|\\
	&\le \limsup_{i\to\infty}|D(f_{r_i})_z|([s-\eps,s+\eps])\\
	&\le |Dw_z|([s-\eps,s+\eps])\\
	&\le 1.
\end{align*}
Note that $w_z(\cdot)=\eta h(\cdot)$, where $\eta$ is a unit vector and $|h|\le 1$.
Thus for large enough $i$ (depending on $z$) we have
\begin{equation}\label{eq:f containment}
		(f_{r_i})^{*}_z([s-\eps,s+\eps])
		\subset Q(0,5).
\end{equation}
By Lemma \ref{lem:strict and uniform conv 2} we also have for large enough $i$
\begin{equation}\label{eq:jump condition 1}
(f_{r_i}^1)^{\vee}_z(t)\ge \eta_1 h^{\vee}(s)-\delta \eta_1 (h^{\vee}(s)-h^{\wedge}(s))
 \quad\textrm{for all }t\in [s+\eps/2,s+\eps]
\end{equation}
and
\begin{equation}\label{eq:jump condition 2}
(f_{r_i}^1)^{\wedge}_z(t)\le \eta_1h^{\wedge}(s)+\delta \eta_1  (h^{\vee}(s)-h^{\wedge}(s))
 \quad\textrm{for all }t\in [s-\eps,s-\eps/2].
\end{equation}
Denote by $E_i\subset D_i$ the set of points $z$ for which
\eqref{eq:f containment},
\eqref{eq:jump condition 1}, and \eqref{eq:jump condition 2} hold.
By Borel measurability, lower semicontinuity, and upper semicontinuity, respectively, the set $E_i$ is 
$\mathcal L^{n-1}$-measurable.
Using \eqref{eq:only small jumps}, note that
\begin{equation}\label{eq:using small jumps}
(f_{r_i}^1)^{\vee}_z(s-\eps/2)
\le \eta_1h^{\wedge}(s)+2\delta \eta_1(h^{\vee}(s)-h^{\wedge}(s)).
\end{equation}
By upper semicontinuity of $t\mapsto (f^1_{r_i})^{\vee}_z(t)$
for $z\in E_i$, we find a minimal
number $t_{i,z}\in [s-\eps/2,s+\eps/2]$ for which
\[
(f^1_{r_i})^{\vee}_z(t_{i,z})\ge \eta_1\frac{ h^{\wedge}(s)+2 h^{\vee}(s)}{3}.
\]
Then by \eqref{eq:two sided continuity} and \eqref{eq:only small jumps}, in fact
\[
\eta_1\frac{h^{\wedge}(s)+2 h^{\vee}(s)}{3}+\delta\eta_1 (h^{\vee}(s)-h^{\wedge}(s))
\ge (f^1_{r_i})^{\vee}_z(t_{i,z})\ge \eta_1\frac{h^{\wedge}(s)+2 h^{\vee}(s)}{3},
\]
and so by \eqref{eq:using small jumps}, necessarily $t_{i,z}>s-\eps/2$.
Then we also find a maximal number
$t'_{i,z}\in (s-\eps/2,t_{i,z})$ for which
\[
(f^1_{r_i})^{\wedge}_z(t_{i,z}')\le \eta_1\frac{2  h^{\wedge}(s)+ h^{\vee}(s)}{3}.
\]
Then in total
\begin{equation}\label{eq:frj bounds at end points}
	(f^1_{r_i})^{\vee}_z(t_{i,z})\ge \eta_1\frac{  h^{\wedge}(s)+2 h^{\vee}(s)}{3}
	\quad\textrm{and}\quad
	(f^1_{r_i})^{\wedge}_z(t_{i,z}')\le \eta_1\frac{2 h^{\wedge}(s)+ h^{\vee}(s)}{3},
\end{equation}
and also
\begin{equation}\label{eq:frj bounds}
\begin{split}
\eta_1\frac{2  h^{\wedge}(s)+ h^{\vee}(s)}{3}-\delta\eta_1 (h^{\vee}(s)-h^{\wedge}(s))
&\le (f^1_{r_i})^{\wedge}_z(t)\\
&\le (f^1_{r_i})^{\vee}_z(t)\\
&\le \eta_1\frac{h^{\wedge}(s)+2h^{\vee}(s)}{3}+\delta\eta_1 (h^{\vee}(s)-h^{\wedge}(s))
\end{split}
\end{equation}
for all $t'_{i,z}\le t\le t_{i,z}$.
By \eqref{eq:fundamental theorem of calculus for BV 2} and \eqref{eq:frj bounds at end points}, we have
\begin{equation}\label{eq:T and tjz}
	\begin{split}
		|D(f_{r_i})_z|([t'_{i,z},t_{i,z}])
		&\ge |(f^1_{r_i})^{\wedge}_z(t'_{i,z})-(f^1_{r_i})^{\vee}_z(t_{i,z})|\\
		&\ge \eta_1 \frac{h^{\vee}(s)-h^{\wedge}(s)}{3}.
	\end{split}
\end{equation}
Due to \eqref{eq:size of Dj sets}, we have
\[
	\lim_{i\to\infty}\mathcal L^{n-1}(E_i)=\mathcal L^{n-1}(Q_{n-1}(0,1)).
\]
Consider the set $F_i\subset E_i$ consisting of points $z$
for which we have $|f^*_{r_i}(z,u_{i,z})-f^*_{r_i}(\widehat{z},\widehat{u})|\ge \delta$ for some
$u_{i,z}\in [t_{i,z}',t_{i,z}]$ and all
\[
(\widehat{z},\widehat{u})\in Q(0,1)\cap r_i^{-1}(U_{x+r_i(z,u_{i,z})}-x)
\quad\textrm{with } \widehat{z}\notin B_{n-1}(z,2\eps).
\]
Recursively, for a fixed $i$ and starting from $l=1$, take points $(z_l,u_l)\in F_i\times [t_{i,z_l}',t_{i,z_l}]$ 
with $u_l=u_{i,z_l}$ such that $z_l\in F_i$ is not in any of the sets
\[
B_{n-1}(z_{\widehat{l}},2\eps)\cup \pi_n[Q_{n-1}(0,1)\setminus r_i^{-1}(U_{x+r_i(z_{\widehat{l}},u_{\widehat{l}})}-x)].
\]
for $\widehat{l}< l$.
Call this collection of points $I$.
By \eqref{eq:U comp is small initial} and \eqref{eq:choice of k prime two}, we have
\[
\#I \ge \frac{\mathcal L^{n-1}(F_i)}{\omega_{n-1}(2\eps)^{n-1}+n^n/k'}
\ge \frac{\mathcal L^{n-1}(F_i)}{2^{n}\omega_{n-1}\eps^{n-1}}.
\]
On the other hand, we have
$|f^*_{r_i}(z_l,u_l)-f^*_{r_i}(z_{\widehat{l}},u_{\widehat{l}})|\ge \delta$
for distinct points $(z_l,u_l)$ and $(z_{\widehat{l}},u_{\widehat{l}})$ in $I$,
and so by \eqref{eq:f containment}, we have
$\# I \le 20^n/(\delta^n\omega_n)$.
By the choice of $\eps$, recall \eqref{eq:choice of eps}, we get
\[
		\mathcal L^{n-1}(F_i)
	\le \frac{1}{4}\mathcal L^{n-1}(B_{n-1}(0,1/10)).
\]
Denote $S_i:=E_i\setminus F_i$.

Thus for large enough $i$, we have
\begin{equation}\label{eq:size of Pj sets}
		\begin{split}
	\mathcal L^{n-1}(B_{n-1}(0,1)\setminus S_i)
	\le \frac{1}{2}\mathcal L^{n-1}(B_{n-1}(0,1/10)).
	\end{split}
\end{equation}
Define
\[
A_i:=
\bigcup_{z\in S_i\cap B_{n-1}(0,1/10)}\{z\}\times [t_{i,z}',t_{i,z}].
\]
By Lemma \ref{lem:last measurability} (proved below) we know that  $S_i\ni z\mapsto t_{i,z}$ and 
$S_i\ni z\mapsto t'_{i,z}$ are Borel functions, and 
so the set $A_i$ is $|D_n f_{r_i}|$-measurable
by Lemma \ref{lem:measurability}.

We also have  for some $s-\eps<\kappa<s-\eps/2$ that 
\[
	(f_{r_i})_z^*(\kappa)\to \eta h(\kappa)
	\quad\textrm{for}\quad
	\mathcal L^{n-1}\textrm{-a.e. }z\in Q_{n-1}(0,1).
\]
Consider the set $H_i$ of those $z\in Q_{n-1}(0,1)$ for which
\[
	|(f_{r_i})_z^*(\kappa)-\eta  h(\kappa)|<\delta \eta_1 (h^{\vee}(s)-h^{\wedge}(s)).
\]
Using \eqref{eq:jump condition 2}, we get
\begin{equation}\label{eq:one third estimate}
	(f_{r_i}^1)_z^*(\kappa) < \eta_1 h^{\wedge}(s)+2\delta \eta_1 (h^{\vee}(s)-h^{\wedge}(s)).
\end{equation}
For sufficiently large $i$, we have
\begin{equation}\label{eq:Hi size}
\mathcal L^{n-1}(Q_{n-1}(0,1)\setminus H_i)<\frac{1}{4}\mathcal L^{n-1}(B(0,\eps/2)).
\end{equation}
For all sufficiently large $i$, we have
\begin{equation}\label{eq:measure of Ai}
\begin{split}
	|D f_{r_i}|(A_i)
	&\ge |D_n f^1_{r_i}|(A_i)\\
	&= \int_{S_i\cap B_{n-1}(0,1/10)}|D (f^1_{r_i})_z|([t_{i,z}',t_{i,z}])\,d\mathcal L^{n-1}(z) 
	\quad\textrm{by }\eqref{eq:slice representation for total variation} \\
	&\ge \mathcal L^{n-1}(S_i\cap B_{n-1}(0,1/10))
	\eta_1 \frac{h^{\vee}(s)-h^{\wedge}(s)}{3}
	\quad\textrm{by }\eqref{eq:T and tjz}\\
	&\ge \mathcal L^{n-1}(B_{n-1}(0,1/10)) \eta_1 \frac{h^{\vee}(s)-h^{\wedge}(s)}{6}
	\quad\textrm{by }\eqref{eq:size of Pj sets}.
\end{split}
\end{equation}
Recall the definition of the scalings \eqref{eq:scalings def}.
For the original function $f$, for all sufficiently large $i$ we get
\begin{align*}
	\frac{|Df|(x+r_i A_i)}{|Df|(Q(x,r_i))}
	&=|D f_{r_i}|(A_i)\\
	&\ge  \mathcal L^{n-1}(B_{n-1}(0,1/10)) \eta_1 \frac{h^{\vee}(s)-h^{\wedge}(s)}{6}
	\quad\textrm{by }\eqref{eq:measure of Ai}.
\end{align*}
Then by \eqref{eq:Cantor density}, for sufficiently large $i$ we also have
\[
\frac{|D^cf|(x+r_i A_i)}{|D^cf|(Q(x,r_i))}
\ge \mathcal L^{n-1}(B_{n-1}(0,1/10)) \eta_1 \frac{h^{\vee}(s)-h^{\wedge}(s)}{7}.
\]

Let $y\in A_i$.
From \eqref{eq:U comp is small}, the set of points
\[
q_y'\in r_i^{-1}(U_{x+r_iy}-x)
\quad\textrm{with}\quad |\pi_n(q_y')-\pi_n(y)|\ge 2\eps
\]
is nonempty.
By the definition of $F_i$, in fact we can find such a $q_y'$ with
\[
	|f^*_{r_i}(q_y')-f^*_{r_i}(y)|<\delta,
\]
and note that $|q_y'-y|\ge 2\eps$.
Combining \eqref{eq:U comp is small initial} and \eqref{eq:choice of k prime two},
we have
\[
		\mathcal L^{n-1}(\pi_n(Q(0,1)\setminus r_i^{-1}(U_{x+r_iy}-x)))
		\le \frac{1}{4}\mathcal L^{n-1}(B(0,\eps/2)),
\]
and then combining with \eqref{eq:Hi size},
we also find a point
\[
q_y:=(z,\kappa)\in r_i^{-1}(U_{x+r_iy}-x)\cap H_i
\]
with $|z-\pi_n(y)|<\eps/2$.
Then $|q_y-y|< 2\eps$.
Using
\eqref{eq:frj bounds} and
\eqref{eq:one third estimate},
we have
\[
|f_{r_i}^*(q_y)-f_{r_i}^*(y)|\ge \eta_1\frac{h^{\vee}(s)-h^{\wedge}(s)}{4}.
\]
Thus
\begin{align*}
H_{f^*_{r_i},r_i^{-1}(U_{x+r_iy}-x)}(y,2\eps)
&\ge \frac{|f^*_{r_i}(q_y)-f^*_{r_i}(y)|}{|f^*_{r_i}(q_y')-f^*_{r_i}(y)|}\\
&> \frac{\eta_1 (h^{\vee}(s)-h^{\wedge}(s))/4}{\delta}\\
&>M+1\quad \textrm{by }\eqref{eq:choice of delta}.
\end{align*}
Now
\[
H_{f^*,U_{x+r_i y}}(x+r_{i}y,2r_i\eps)
\ge H_{f^*_{r_i},r_i^{-1}(U_{x+r_iy}-x)}(y,2\eps)
>M+1.
\]
It follows that
\begin{align*}
H_f^{\textrm{fine},k,k'}(x+r_iy)+1	
&\ge \sup_{0<r\le 1/k}H_{f^*,U_{x+r_i y}}(x+r_{i}y,r)\quad\textrm{by }\eqref{eq:almost optimality}\\
&\ge H_{f^*,U_{x+r_i y}}(x+r_{i}y,2\eps r_{i})\\
&>M+1.
\end{align*}
Apply Lemma \ref{lem:nonmeasurable} with the choices $\mu=|D^c f|$
and $D$ is the set where $Df$ has rank one and
$H_f^{\textrm{fine},k,k'}\le M$.
By the lemma, this set has $|D^c f|$-measure zero.
Thus $H_f^{\textrm{fine},k,k'}(x)>M$ for $|D^c f|$-a.e. $x\in \R^n$,
which is what we needed to show (recall \eqref{eq:reduction}).\\

\textbf{The jump part.}
By Alberti's rank one theorem, $\tfrac{dDf}{d|Df|} (x)$ is of rank one for $|D^j f|$-a.e. $x\in\R^n$,
and so we need to show that $H_{f^*}^{\textrm{fine}}(x)=\infty$.
Recall definitions from Section \ref{sec:BV functions}.
Consider the sets
\[
J_f\cap \{f^{-}\in B(q_j,r_k)\}\cap \{f^{+}\in B(q_l,r_m)\}.
\]
for $q_j,q_l\in \Q^n$ and $r_k,r_m\in \Q_+$, $|q_j-q_l|>r_k+r_m$.
Each of these sets has finite $\mathcal H^{n-1}$-measure due to \eqref{eq:jump part representation},
and is countably $\mathcal H^{n-1}$-rectifiable (see \cite[Theorem 3.78]{AFP}).
Thus for each of these sets,
from Lemma \ref{lem:rectifiable and cap} we obtain an exceptional set of $\mathcal H^{n-1}$-measure zero.
Since there are countably many sets,
corresponding to different choices of
$q_j,q_l,r_k,r_m$,
we get countably many exceptional sets of $\mathcal H^{n-1}$-measure zero.

Fix $x\in J_f$ which is outside all of the exceptional sets;
this is true of $\mathcal H^{n-1}$-a.e. $x\in J_f$ and thus of $|D^j f|$-a.e. $x\in\R^n$.
Fix $0<\eps<|f^{-}(x)-f^{+}(x)|/2$.
Choose $q_j,q_l,r_k,r_m$ such that $r_k+r_m<\eps$ and $r_k+r_m<|q_j-q_l|$, and
\[
f^{-}(x)\in B(q_j,r_k)\quad\textrm{and}\quad f^{+}(x)\in B(q_l,r_m).
\]
Now $x$ is contained in
\[
S:=J_f\cap \{f^{-}\in B(q_j,r_k)\}\cap \{f^{+}\in B(q_l,r_m)\}.
\]
Let $U$ be an arbitrary $1$-finely open set containing $x$.
From Lemma \ref{lem:rectifiable and cap}, we have
\[
\lim_{r\to 0}\frac{\mathcal H^{n-1}(B(x,r)\cap S\cap U)}{\omega_{n-1} r^{n-1}}=1.
\]
In particular, we find a sequence of points $y_i\in S\cap U$
with $y_i\to x$, $y_i\neq x$.
By the definition of $S$ and by \eqref{eq:jump point formula}, we have
\begin{align*}
	|f^*(y_i)-f^*(x)|
	&=\left|\frac{f^{-}(y_i)+f^{+}(y_i)}{2}-\frac{f^{-}(x)+f^{+}(x)}{2}\right|\\
	&\le \frac{1}{2}|f^{-}(y_i)-f^{-}(x)|+\frac{1}{2}|f^{+}(y_i)-f^{+}(x)|\\
	&\le r_k+r_m\le \eps.
\end{align*}
By Lemma \ref{lem:capa in small ball} we know that
\[
\lim_{r\to 0}\frac{\mathcal L^n(B(x,r)\setminus U)}{\mathcal L^n(B(x,r))} = 0.
\]
Now from the asymptotic behavior given in \eqref{eq:jump value 1} and
\eqref{eq:jump value 2}, for any $r>0$ we find a sequence of 
points $\widehat{y}_i\in B_{\nu}^+(x,r)\cap U$ converging to $x$ with
$|f^*(\widehat{y}_i)-f^+(x)|<\eps$, and then
\begin{align*}
	|f^*(\widehat{y}_i)-f^*(x)|
	&=|f^*(\widehat{y}_i)-(f^{-}(x)+f^+(x))/2|\\
	&\ge \frac{1}{2}|f^{+}(x)-f^{-}(x)|-|f^*(\widehat{y}_i)-f^{+}(x)|\\
	&\ge \frac{1}{2}|f^{+}(x)-f^{-}(x)|-\eps.
\end{align*}
Hence for all sufficiently small $r>0$,
\[
H_{f^*,U}(x,r)=\frac{L_{f^*,U}(x,r)}{l_{f^*,U}(x,r)}\ge \frac{\frac{1}{2}|f^{+}(x)-f^{-}(x)|-\eps}{\eps}.
\]
We get
\[
\limsup_{r\to 0}H_{f^*,U}(x,r)
\ge \frac{\frac{1}{2}|f^{+}(x)-f^{-}(x)|-\eps}{\eps}.
\]
Since $\eps>0$ was arbitrary and so was the $1$-finely open set $U$, we get $H_{f^*}^{\textrm{fine}}(x)=\infty$.
\end{proof}

\begin{lemma}\label{lem:last measurability}
	The functions $S_i\ni z\mapsto t_{i,z}$ and 
	$S_i\ni z\mapsto t'_{i,z}$ in Case 2 of the proof of the Cantor part of Theorem
	\ref{thm:rank} are Borel functions.
\end{lemma}
\begin{proof}
	We have that $(f^1_{r_i})^{\vee}$ restricted to $S_i\times (-1/2,1/2)$ 
	is upper semicontinuous, which implies that $S_i\ni z\mapsto t_{i,z}$ is lower semicontinuous.
	The subgraph
	\[
	\{(z,t)\colon z\in S_i,\,s-\eps/2<t<t_{i,z}\}
	\] 
	is then relatively open in $S_i\times (s-\eps/2,s+\eps/2)$
	and can be represented as a countable union of closed
	cylinders
	\[
	C_k:=\overline{B}(x_k,r_k)\times [a_k,b_k],\quad k\in\N,
	\]
	intersected with $S_i\times (s-\eps/2,s+\eps/2)$.
	Given $z\in \overline{B}(x_k,r_k)\cap S_i$, let $t_{i,z,k}'$ be the maximal number
	in $(s-\eps/2,b_k]$ for which
	\[
	(f^1_{r_i})^{\wedge}_z(t_{i,z,k}')\le \eta_1\frac{2  h^{\wedge}(s)+ h^{\vee}(s)}{3}.
	\]
	Since $(f^1_{r_i})^{\wedge}$ restricted to $S_i$ 
	is lower semicontinuous, the mapping 
	$ \overline{B}(x_k,r_k)\cap S_i\ni z\mapsto t_{i,z,k}'$ is upper semicontinuous.
	Then
	\[
	S_i\ni z\mapsto t_{i,z}'=\sup_{k\in\N} t_{i,z,k}'
	\]
	is Borel.
\end{proof}


\begin{thebibliography}{ACMM}

\bibitem{Al}G. Alberti,
\textit{Rank one property for derivatives of functions with bounded variation},
Proc. Roy. Soc. Edinburgh Sect. A 123 (1993), no. 2, 239--274.

\bibitem{AFP}L. Ambrosio, N. Fusco, and D. Pallara,
\textit{Functions of bounded variation and free discontinuity problems.}
Oxford Mathematical Monographs. The Clarendon Press, Oxford University Press, New York, 2000.

\bibitem{BB}A. Bj\"orn and J. Bj\"orn,
\textit{Nonlinear potential theory on metric spaces},
EMS Tracts in Mathematics, 17. European Mathematical Society (EMS), Z\"urich, 2011. xii+403 pp.

\bibitem{CDLP}M. Carriero, G. Dal Maso, A. Leaci, and E. Pascali,
\textit{Relaxation of the nonparametric plateau problem with an obstacle},
J. Math. Pures Appl. (9) 67 (1988), no. 4, 359--396.

\bibitem{EvGa}L. C. Evans and R. F. Gariepy,
\textit{Measure theory and fine properties of functions.}
Revised edition. Textbooks in Mathematics. CRC Press, Boca Raton, FL, 2015. xiv+299 pp. 

\bibitem{HKST}J. Heinonen, P. Koskela, N. Shanmugalingam, and J. Tyson,
\textit{Sobolev classes of Banach space-valued functions and quasiconformal mappings},
J. Anal. Math. 85 (2001), 87--139.

\bibitem{EvGa}L. C. Evans and R. F. Gariepy,
\textit{Measure theory and fine properties of functions.}
Revised edition. Textbooks in Mathematics. CRC Press, Boca Raton, FL, 2015. xiv+299 pp.

\bibitem{FoLe}I. Fonseca and G. Leoni,
\textit{Modern Methods in the Calculus of Variations},
$L^p$ Spaces. Springer, Berlin (2007).

\bibitem{L-FC}P. Lahti,
\textit{A notion of fine continuity for BV functions on metric spaces},
Potential Anal. 46 (2017), no. 2, 279--294.

\bibitem{L-FC}P. Lahti,
\textit{Finely quasiconformal mappings},
preprint 2024.

\bibitem{L-SA}P. Lahti,
\textit{Strong approximation of sets of finite perimeter in metric spaces},
Manuscripta Math. 155 (2018), no. 3-4, 503--522.

\bibitem{L-CK}P. Lahti,
\textit{The Choquet and Kellogg properties for the fine topology when $p=1$ in metric spaces},
J. Math. Pures Appl. (9) 126 (2019), 195--213.

\bibitem{LaSh}P. Lahti and N. Shanmugalingam,
\textit{Fine properties and a notion of quasicontinuity for $\BV$ functions on metric spaces},
J. Math. Pures Appl. (9) 107 (2017), no. 2, 150--182.



\end{thebibliography}
\end{document}